%% file: Pointwise_definable_end_extensions.tex
\documentclass{amsart}

\title[Pointwise-definable end extensions]{Every countable model of arithmetic or set theory has a pointwise-definable end extension}

\author{Joel David Hamkins}
\address[Joel David Hamkins]
{O'Hara Professor of Logic, University of Notre Dame, 100 Malloy Hall, Notre Dame, IN 46556 USA \&\ Associate Faculty Member, Professor of Logic, Faculty of Philosophy, University of Oxford, UK}
\email{jdhamkins@nd.edu}

\urladdr{http://jdh.hamkins.org}

\thanks{This article was the subject of my talk at the Salzburgiense Concilium Omnibus Philosophis Analyticis (SOPhiA 2022) in the workshop, ``Reflecting on ten years in the set-theoretic multiverse.'' I am grateful to the organizers and participants, and also to Victoria Gitman, Roman Kossak, and Ali Enayat for helpful comments. Commentary can be made about this article on my blog at \href{http://jdh.hamkins.org/pointwise-definable-end-extensions}{http://jdh.hamkins.org/pointwise-definable-end-extensions}.}
\usepackage{latexsym,amsfonts,amsmath,amssymb,mathrsfs}
\usepackage{bbm}
\usepackage[hidelinks]{hyperref}
\usepackage[cmyk,svgnames,dvipsnames]{xcolor}
\usepackage{tikz}
\usetikzlibrary{arrows,arrows.meta,petri,topaths,positioning,shapes,shapes.misc,patterns,calc,decorations.pathreplacing,hobby}
\pgfdeclarelayer{foreground}
\pgfdeclarelayer{background}
\pgfdeclarelayer{boardarrows}
\pgfdeclarelayer{boardgrid}
\pgfdeclarelayer{boardshades}
\pgfsetlayers{boardshades,boardgrid,boardarrows,background,main,foreground}
\usepackage{wrapfig} % for figures with wrapping text
\usepackage{caption} % for caption positioning
\DeclareCaptionStyle{compact}{font=footnotesize,format=plain,name={Fig},labelsep=period,skip=1.5ex,margin=0pt,oneside,justification=centering}
\DeclareCaptionStyle{leftside}{style=compact,margin={0pt,9pt}}
\DeclareCaptionStyle{rightside}{style=compact,margin={9pt,0pt}}
\usepackage{float}
\usepackage[utf8]{inputenc}
\RequirePackage{doi}

\usepackage{enumitem}
\include{MathMacrosJDH}
\usepackage[backend=bibtex,style=alphabetic,maxbibnames=15,maxcitenames=6,dateabbrev=false]{biblatex}
\bibliography{HamkinsBiblio,MathBiblio,PhilBiblio,WebPosts}
\renewcommand{\UrlFont}{} % makes url text smaller (used only in bibliography?)
\renewbibmacro{in:}{\ifentrytype{article}{}{\printtext{\bibstring{in}\intitlepunct}}} % removes "In:" in article format.
\DeclareFieldFormat{url}{{\UrlFont\url{#1}}} % removes redundant "URL:" text, and just gives the url.
\DeclareFieldFormat{urldate}{% improves setting of urldate
  (version \thefield{urlday}\addspace%
  \mkbibmonth{\thefield{urlmonth}}\addspace%
  \thefield{urlyear}\isdot)}
\DeclareFieldFormat{eprint:arxiv}{%  JDH format for arxiv info in bibliography
\ifhyperref
    {\href{http://arxiv.org/abs/#1}{%
    arXiv\addcolon\nolinkurl{#1}}\iffieldundef{eprintclass}{}{\UrlFont{\mkbibbrackets{\thefield{eprintclass}}}}}
    {arXiv\addcolon\nolinkurl{#1}\iffieldundef{eprintclass}{}{\UrlFont{\mkbibbrackets{\thefield{eprintclass}}}}}}
\newcommand\ZFbar{\overline\ZF}
\newcommand\ZFCbar{\overline\ZFC}
\newcommand{\PAp}{\PA^{\scriptscriptstyle\!+}}

\begin{document}

\begin{abstract}
According to the math tea argument, there must be real numbers that we cannot describe or define, because there are uncountably many real numbers, but only countably many definitions. And yet, the existence of pointwise-definable models of set theory, in which every individual is definable without parameters, challenges this conclusion. In this article, I introduce a flexible new method for constructing pointwise-definable models of arithmetic and set theory, showing furthermore that every countable model of Zermelo-Fraenkel \ZF\ set theory and of Peano arithmetic \PA\ has a pointwise-definable end extension. In the arithmetic case, I use the universal algorithm and its $\Sigma_n$ generalizations to build a progressively elementary tower making any desired individual $a_n$ definable at each stage $n$, while preserving these definitions through to the limit model, which can thus be arranged to be pointwise definable. A similar method works in set theory, and one can moreover achieve $V=L$ in the extension or indeed any other suitable theory holding in an inner model of the original model, thereby fulfilling the resurrection phenomenon. For example, every countable model of \ZF\ with an inner model with a measurable cardinal has an end extension to a pointwise-definable model of $\ZFC+V=L[\mu]$.
\end{abstract}

\maketitle

\section{Introduction}

A model in first-order logic is \emph{pointwise definable} if every individual is definable in the model without parameters---every individual has an expressible property that only that individual has. A pointwise-definable model is thus completely determined by its theory, the set of sentences true in that model, for the existence of the definable elements and their atomic structural features, including the manner in which they relate to one another, are all to be found as particular assertions in the theory. One can thus reconstruct an isomorphic copy of the model using only the syntactic information in the theory.

One sometimes hears a certain philosophical argument, which I have called the \emph{math tea} argument (\cite{HamkinsLinetskyReitz2013:PointwiseDefinableModelsOfSetTheory}), for it might be heard in the contemplative discussions at a good math tea, namely, that there must be some real numbers that we can neither describe nor define, because there are only countably many definitions, but uncountably many real numbers. Although the argument may seem initially solid---mathematicians often readily agree with it---nevertheless the existence of pointwise-definable models of set theory presents a troublesome counterpoint, revealing the argument to be necessarily subtler than usually presented. A serious engagement with the math tea argument, in my view, leads to deep metamathematical issues concerning the nature of truth and definability.

In a pointwise-definable model of Zermelo-Fraenkel \ZFC\ set theory, after all, the real numbers form an uncountable set, since this is a theorem of \ZFC, and yet every real number, every function, every topological space, every set altogether, is uniquely characterized in such a model by some defining set-theoretical property. The existence of pointwise-definable models of set theory thus reveals that one cannot undertake the math tea argument as an ordinary mathematical argument, one formalized in \ZFC, for then it should work even in such a pointwise-definable context, but it cannot since every real number is definable in such a world. The math tea argument thus must take place outside mathematics, a metamathematical argument engaging with the ineffable nature of definability.

Let me explain exactly how the argument goes awry in pointwise-definable models. The basic issue is that although there are indeed only countably many definitions, a pointwise-definable model is not able to form the map associating each definition with the real number that it defines, and so it is unable to perform the diagonalization against this enumeration that would result in a nondefinable real number. One cannot express the concept ``definition $\varphi$ defines object $x$'' as a binary relation ranging over definitions and objects, since from this concept one could define a truth predicate, violating Tarski's theorem on the non-definability of truth. For example, sentence $\sigma$ is true if and only if the formula $\sigma\wedge\forall y\,(y\notin x)$ is a defining property of the emptyset, $x=\emptyset$. In this sense, one cannot even begin to undertake the math tea argument in the object theory of set theory. See further discussion in \cite{Hamkins.MO44102:AnalysisInFactAnalysisOfDefinableNumbers, HamkinsLinetskyReitz2013:PointwiseDefinableModelsOfSetTheory}

In light of this, I find the existence of pointwise-definable models of arithmetic and set theory to be a rich topic, both mathematically and philosophically, shedding light on fundamental aspects of definability and truth. They can be seen as particularly dramatic instances of the Skolem paradox, by which notions of countability, uncountability, and even finiteness need not be absolute between set-theoretic contexts.

There is in fact a robust phenomenon of pointwise-definable models in both arithmetic and set theory. Because \PA\ offers definable Skolem functions (take the least witness), it follows that the definable elements in any model of \PA\ will form an elementary substructure of the model, and this substructure will be pointwise definable because the same definitions work in the submodel as in the parent model. Every consistent theory extending \PA\ therefore gives rise to a (unique) pointwise-definable model, and indeed, the pointwise-definable models of \PA\ are exactly the prime models of this theory, those which embed elementarily into all other models of the same theory.

\enlargethispage{20pt}%
A similar analysis works in set theory, but one must add the axiom $V=\HOD$, of course, since every pointwise-definable model of set theory necessarily satisfies it, as one doesn't even need ordinal parameters in the definitions. Pointwise definability can thus be seen as a strong form of the global choice principle, since $V=\HOD$ is equivalent to the existence of a definable global well ordering of the universe. Just as in arithmetic, this leads again to definable Skolem functions, and so the definable elements of any model of $\ZFC+V=\HOD$ form a pointwise-definable elementary submodel. Indeed, the pointwise-definable models of set theory are exactly the prime models of this theory. The observation that there are pointwise-definable models of \Godel-Bernays set theory goes back to John Myhill \cite{Myhill1952:The-hypothesis-that-all-classes-are-nameable}.\goodbreak

Meanwhile, I find it remarkable that we can also achieve pointwise definability in outer models of a given model of set theory. In \cite{HamkinsLinetskyReitz2013:PointwiseDefinableModelsOfSetTheory}, for example, following on observations of Ali Enayat \cite{Enayat2005:ModelsOfSetTheoryWithDefinableOrdinals}, my co-authors and I proved that every countable model of \ZFC\ and indeed of \GBC\ has a pointwise-definable end extension. The method first adds a proper class $U$, in the style of \cite{Simpson1974:ForcingAndModelsOfArithmetic}, from which every individual is definable, and then forces to code $U$, as well as all the new sets added by the forcing itself, into the \GCH\ pattern. One can also accommodate arbitrary classes into the constructions, after first expanding the model to one that is \emph{principal}, meaning that there is one class from which all others are definable.

In this article, I present a flexible new method for achieving pointwise definability in covering end extensions of a given countable model $M$. Relying on tools from the universal algorithm and its generalizations, the new method furthermore enables stronger conclusions. In particular, for both arithmetic and set theory, we can realize pointwise definability not just in an end extension but in a $\Sigma_m$-elementary end extension (this requires $V=\HOD$ in the set theory case). In addition, every countable model of \ZF\ has pointwise-definable end extensions satisfying $V=L$ or indeed any desired theory holding in an inner model of $V=\HOD$.

To clarify some terminology, an \emph{end-extension} of a model $M$ of \PA\ is an extension $M\of N$ in which all new elements of $N$ are above every element of $M$. Similarly, an end extension of a model of set theory $M$ is a model $N$ of which $M$ is a submodel, such that $N$ adds no new elements to old sets, that is, $a\in^N b\in M\implies a\in M$. Such an end extension $N$ is \emph{covering} $M$, if there is a set $b\in N$ that contains every element of $M$ as an $\in^N$ element, meaning $a\in^N b$ for all $a\in M$. The extension $N$ is a \emph{top extension}, also known as a \emph{rank extension}, if all new sets of $N$ have rank exceeding every ordinal of $M$, so that $V_\alpha^M=V_\alpha^N$ for ordinals $\alpha$ in $M$.

For models of arithmetic the main theorem is the following:

\begin{maintheorem}\
   \begin{enumerate}
     \item Every countable model of \PA\ has a pointwise-definable end extension satisfying \PA.
     \item Indeed, for any particular $m$, there is a $\Sigma_m$-elementary such pointwise-definable end extension.
   \end{enumerate}
\end{maintheorem}

And for models of set theory the main theorem is the following:

\begin{maintheorem}\
  \begin{enumerate}
    \item Every countable model of \ZF\ has a pointwise-definable end extension satisfying $\ZFC+V=L$.
    \item More generally, every countable model of \ZF\ with an inner model of a c.e. theory $\ZFCbar$ that includes $\ZFC+V=\HOD$ has a pointwise-definable end extension satisfying $\ZFCbar$.
    \item Indeed, for any particular $m$, every countable model of $\ZFC+V=\HOD$ admits a pointwise-definable $\Sigma_m$-elementary end extension. When $m\geq 1$, these are top extensions.
  \end{enumerate}
\end{maintheorem}

This theorem generalizes the Barwise extension theorem and other results in \cite{Barwise1971:InfinitaryMethodsInTheModelTheoryOfSetTheoryLC69} and \cite{SuzukiWilmbers1973:Nonstandard-models-for-set-theory}. Barwise had proved that every countable model of \ZF\ admits an end-extension to a model of $\ZFC+V=L$, showing also that certain theories $\ZFCbar$ extending $\ZFC+V=\HOD$ can be obtained in the extension, if they are true in the original model. My theorem strengthens this to the resurrection claim of statement (2), which requires only that $\ZFCbar$ holds in an inner model of the original model. Barwise \cite[theorem~4.2]{Barwise1971:InfinitaryMethodsInTheModelTheoryOfSetTheoryLC69} proved that countable $\omega$-standard models admit pointwise definable extensions, and Suzuki and Wilmers \cite{SuzukiWilmbers1973:Nonstandard-models-for-set-theory} obtain this with top extensions. My theorem strengthens both these results to omit the $\omega$-standard requirement---it applies to all countable models. Barwise found the desired extensions, which furthermore satisfied the same $\Sigma_m$ theory as the original model. My theorem strengthens this to $\Sigma_m$-elementarity, allowing parameters from the ground model. When $m\geq 1$, this implies that the extension is a top extension.

Perhaps the principal contribution of this article, however, is the flexible new construction method, based on the universal algorithm. This is a totally different proof idea than what appears in the earlier results. I shall prove both of these theorems and afterwards draw some philosophical conclusions.

\section{Review of the universal algorithm}

The proof of the main theorem will rely on generalizations of the universal algorithm, provided in a remarkable theorem of W. Hugh Woodin. So before proving the main theorems, let me begin with a review of the universal algorithm theorem.\goodbreak

The universal algorithm theorem concerns a Turing machine program $e$, viewed as a computable procedure for enumerating a sequence of numbers
 $$a_0, a_1, \ldots, a_n.$$
The operation of such a computable process is arithmetically expressible---indeed, it is described by a $\Sigma_1$ formula---and so it makes sense to speak of running a given program inside any model of \PA. Even a nonstandard model of \PA\ can make perfect sense of running the program for a certain number of steps, even a nonstandard number of steps, giving output, placing another number onto the enumerated sequence, and so on. All of the usual computational features of such a computational process are arithmetically expressible and thus interpretable inside any model of \PA.

\enlargethispage{20pt}%
\begin{theorem}[\cite{Woodin2011:A-potential-subtlety-concerning-the-distinction-between-determinism-and-nondeterminism, BlanckEnayat2017:Marginalia-on-a-theorem-of-Woodin, Hamkins.blog2017:The-universal-algorithm-a-new-simple-proof-of-Woodins-theorem, Hamkins:The-modal-logic-of-arithmetic-potentialism}]\label{Theorem.Universal-algorithm} For any consistent c.e. theory $\PAp$ extending \PA, there is a Turing machine program $e$ for computably enumerating a sequence with the following properties.
  \begin{enumerate}
    \item $\PA$ proves that the sequence enumerated by $e$ is finite. And it is empty when computed in the standard model $\N$.
    \item For any model $M\satisfies\PAp$, if $s$ is the sequence enumerated by $e$ in $M$ and $t$ is any finite extension of $s$ in $M$, then there is an end extension of $M$ to a model $N\satisfies\PAp$ in which $e$ enumerates exactly $t$.
  \end{enumerate}

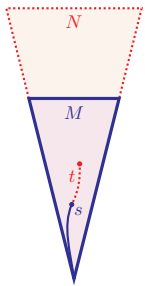
$$\begin{tikzpicture}[xscale=.6,scale=.6]
% extension N
 \draw[thick,red,densely dotted,fill=red!5] (-2,4) -- (-3,6) -- node[below,scale=.7] {$N$} (3,6) -- (2,4);
% M
 \draw[very thick,blue,fill=Orchid!20] (0,0) -- (-2,4) --  node[below,blue,scale=.7] {$M$} (2,4) -- cycle;
% draw s & t
 \draw[thick,blue] (0,0) to[out=105,in=-130] (-.1,1.65) node[blue,circle,fill=blue,scale=.2,label={[scale=.7,label distance=-5pt]below right:$s$}] (s) {};
 \draw[thick,densely dotted,red] (s) to[out=50,in=-85] (.25,2.55) node[red,circle,fill=red,scale=.2,label={[scale=.7,label distance=-3pt]185:$t$}] {};
\end{tikzpicture}$$
\end{theorem}\goodbreak

\newpage
In brief, the program defines a sequence with the \emph{universal extension property} in models of \PA. The history of the universal algorithm result involves several instances of independent rediscovery of closely related or essentially similar results, alternative arguments and generalizations. Woodin's original theorem appears in~\cite{Woodin2011:A-potential-subtlety-concerning-the-distinction-between-determinism-and-nondeterminism}, with further results by Blanck and Enayat~\cite{BlanckEnayat2017:Marginalia-on-a-theorem-of-Woodin, Blanck2017:Dissertation:Contributions-to-the-metamathematics-of-arithmetic}, including, importantly, the extension of the result from countable to arbitrary models. I provided a simplified proof in \cite{Hamkins:The-modal-logic-of-arithmetic-potentialism, Hamkins.blog2017:The-universal-algorithm-a-new-simple-proof-of-Woodins-theorem}. Shavrukov reportedly had found a similar proof independently in 2012 (unpublished), pointing out the affinity with Berarducci-Japaridze~\cite{Japaridze1994:A-simple-proof-of-arithmetical-completeness-for-Pi11-conservativity-logic}. Visser has pointed out a similar affinity with the `refugee' argument in \cite{ArtemovBeklemishev2004:Provability-logic}. Weaker incipient forms of the theorem are due to Mostowski~\cite{Mostowski1960:A-generalization-of-the-incompleteness-theorem} and Kripke~\cite{Kripke1962:Flexible-predicates-of-formal-number-theory}.

In my article \cite[theorem~18]{Hamkins:The-modal-logic-of-arithmetic-potentialism}, I proved the following generalization of the universal algorithm to higher levels of complexity; a version is also proved independently by Blanck \cite{Blanck2017:Dissertation:Contributions-to-the-metamathematics-of-arithmetic}. Specifically, there is a $\Sigma_{m+1}$ definable finite sequence that exhibits the universal extension property in models of $\PAp$ with respect to $\Sigma_m$ elementary end extensions. The case $m=0$ is exactly the universal algorithm result of theorem \ref{Theorem.Universal-algorithm}, because a computable enumerable set is $\Sigma_1$ definable and every end extension is $\Sigma_0$ elementary.

\begin{theorem}\label{Theorem.Universal-algorithm-Sigma-m}
For any consistent c.e. theory $\PAp$ extending \PA\ and any standard-finite natural number $m$, there is an oracle Turing machine program $\tilde e$ with the following properties.
  \begin{enumerate}
    \item $\PA$ proves that the sequence enumerated by $\tilde e$, with any oracle, is finite. And it is empty when computed in the standard model $\N$ with oracle $0^{(n)}$.
    \item For any model $M$ of $\PAp$ in which $\tilde e$, using oracle $0^{(m)}$ of $M$, enumerates a finite sequence $s$, possibly nonstandard, and any $t\in M$ extending $s$, there is a $\Sigma_m$-elementary end extension of $M$ to a model $N\satisfies\PAp$ in which $\tilde e$, with oracle $0^{(m)}$ of $N$, enumerates exactly $t$.
  \end{enumerate}
\end{theorem}

In \cite{Hamkins:The-modal-logic-of-arithmetic-potentialism} it is stated that the sequence is empty in the standard model, provided that $\N\satisfies\PAp$, but in fact this extra provisional assumption is not needed for that conclusion, since if there were a successful stage in the standard model, then it would also succeed in every model of $\PAp$, which would show that every model of $\PAp$ had behavior contradictory to the proof that was found in the standard model, contradicting consistency.

\section{Pointwise-definable end extensions in arithmetic}

Using the results of the previous section, let me now prove main theorem 1, which will provide the new flexible method of constructing pointwise-definable end extensions.

\begin{theorem}[Main theorem 1]\label{Theorem.Main-theorem-1-numbered}
  Every countable model of \PA\ has a pointwise-definable end extension satisfying \PA. Indeed, for any particular $m$, there is a $\Sigma_m$-elementary such pointwise-definable end extension.
\end{theorem}

\newpage\hskip0pt
\begin{wrapfigure}{r}{.25\textwidth}\vskip2ex\hfill
\begin{tikzpicture}[xscale=.6,scale=.5]
% whole model
 \draw[densely dotted,dotted] (0,0) -- (-4.5,9) -- node[below,scale=1.1] {$\vdots$} (4.5,9) -- cycle;
% extension M3
 \draw[very thin,blue,fill=Orchid!5] (0,0) -- (-3.5,7) --  node[below,blue,scale=.5] {$M_3$} (3.5,7) -- cycle;
% extension M2
 \draw[thin,blue,fill=Orchid!10] (0,0) -- (-3,6) --  node[below,blue,scale=.6] {$M_2$} (3,6) -- cycle;
% extension M1
 \draw[thick,blue,fill=Orchid!25] (0,0) -- (-2.5,5) --  node[below,blue,scale=.7] {$M_1$} (2.5,5) -- cycle;
% M0
 \draw[very thick,blue,fill=Orchid!40] (0,0) -- (-1.5,3) --  node[below,blue,scale=.7] {$M_0$} (1.5,3) -- cycle;
\end{tikzpicture}
\end{wrapfigure}
\hskip0pt
\begin{proof}
Fix any countable model $M_0$ of \PA. By the universal algorithm, we can find a countable end extension model $M_1$ where the last entry of the universal algorithm is any desired object $a_0$ of $M_0$. Next, by theorem \ref{Theorem.Universal-algorithm-Sigma-m} we may find a countable $\Sigma_1$-elementary end extension $M_2$, such that the last element on the universal $\Sigma_2$-sequence is any desired element $a_1$ of $M_1$. Because the extension is $\Sigma_1$-elementary, the operation of the universal algorithm is frozen---no new numbers will appear on the sequence---and so $a_0$ is still last on that sequence. And so on. At each stage $n$, we find by theorem \ref{Theorem.Universal-algorithm-Sigma-m} a countable $\Sigma_n$-elementary end extension $M_{n+1}$ of $M_n$, such that any desired $a_n$ from $M_n$ is last on the $\Sigma_{n+1}$-definable universal sequence, and this will be preserved to the later models. The result is a progressively elementary tower of models:
 $$M_0\quad\elesub_{\Sigma_0}\quad M_1\quad\elesub_{\Sigma_1}\quad M_2\quad\elesub_{\Sigma_2}\quad M_3\quad\elesub_{\Sigma_3}\quad M_4\quad\elesub_{\Sigma_4}\quad \cdots$$
Because the tower is progressively elementary, each model $M_n$ will be $\Sigma_n$ elementary in the limit model $M^*=\Union_n M_n$, which will therefore be a model of \PA. Every $a_n$ is definable in $M^*$ as the last element on the $\Sigma_{n+1}$-definable universal sequence, since this is true in $M_{n+1}$ and was preserved to all the later extensions and the limit. Therefore, by undertaking a simple bookkeeping method we can arrange that the elements $a_n$ exhaust $M^*$. We should simply arrange that every element appearing in any of the models $M_k$ is eventually made definable as one of the objects $a_n$. In this way, $M^*$ becomes pointwise definable, as desired.

To achieve a $\Sigma_m$-elementary end extension, we need simply start right in with the $\Sigma_{m+1}$-definable sequence at stage $0$, and then proceed to $\Sigma_{m+2}$, and so on.
 $$M_0\quad\elesub_{\Sigma_m}\quad M_1\quad\elesub_{\Sigma_{m+1}}\quad M_2\quad\elesub_{\Sigma_{m+2}}\quad M_3\quad\elesub_{\Sigma_{m+3}}\quad \cdots$$
Every stage in this tower will be at least $\Sigma_m$-elementary.
\end{proof}\goodbreak

When I mentioned this result to Roman Kossak, he pointed out an alternative manner of proving main theorem 1 using only classical results in the theory of models of arithmetic rather than the universal algorithm. To outline the argument briefly, one begins with any countable model $M$, forms the $\Sigma_1$ theory of $M$, and then applies the method of \cite{JensenEhrenfeucht1976:Some-problems-in-elementary-arithmetics} using what they call $B_n$ independent trees to construct a prime model of that theory with the same standard system as $M$, which by the Friedman embedding theorem \cite{Friedman1973:CountableModelsOfSetTheories} will have $M$ as an initial segment. And so $M$ has a pointwise-definable end extension.

Because the universal algorithm also works with uncountable models, the same method shows that for any model of arithmetic $M\satisfies\PA$, of any cardinality, and any countably many elements of $M$, there is for any particular $m$ a $\Sigma_m$ elementary end extension $N$ in which those elements, individually, are definable. In forthcoming work \cite{GitmanHamkins:Every-model-of-arithmetic-or-set-theory-size-at-most-continuum-has-Leibnizian-extension}, Victoria Gitman and I use this aspect of the construction to show that every model of arithmetic of size at most continuum admits a Leibnizian extension, one in which any two elements are distinguished by some expressible property. And similarly in set theory.

\section{Set-theoretic analogues of the universal sequence}

W. Hugh Woodin and I had generalized the universal algorithm result to models of set theory, providing a $\Sigma_2$-definable finite sequence with the universal extension property with respect to top extensions of countable models of set theory.

\begin{theorem}[Hamkins, Woodin\cite{HamkinsWoodin:The-universal-finite-set}]\label{Theorem.Sigma_2-universal}
There is a $\Sigma_2$ definable finite sequence of sets
$$a_0,\ a_1,\ \ldots,\ a_n$$
with the following properties:
\begin{enumerate}
  \item \ZFC\ proves that the sequence is finite, and it is empty in any transitive model of \ZFC.
  \item If $M$ is a countable model of \ZFC\ in which the sequence is $s$ and $t$ is any finite extension of $s$ in $M$, then there is a top extension of $M$ to a model $N\satisfies\ZFC$ in which $\varphi$ defines exactly $t$.
\end{enumerate}
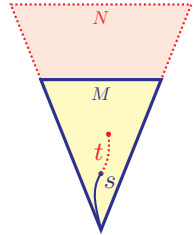
$$\begin{tikzpicture}[xscale=.8,scale=.5]
% extension N
 \draw[thick,red,densely dotted,fill=red!10] (-2,4) -- (-3,6) -- node[below,scale=.7] {$N$} (3,6) -- (2,4);
% M
 \draw[very thick,blue,fill=Yellow!30] (0,0) -- (-2,4) --  node[below,blue,scale=.7] {$M$} (2,4) -- cycle;
% draw s & t
 \draw[thick,blue] (0,0) to[out=105,in=-130] (0,1.5) node[blue,circle,fill=blue,scale=.2,label={[label distance=-5pt]below right:$s$}] {};
 \draw[thick,densely dotted,red] (0,1.5) to[out=50,in=-85] (.25,2.55) node[red,circle,fill=red,scale=.2,label={[label distance=-3pt]185:$t$}] {};
\end{tikzpicture}$$
\end{theorem}

I have stated the theorem here in the sequence formulation, whereas the main theorem of \cite{HamkinsWoodin:The-universal-finite-set} is stated in terms of a universal finite set, but the two formulations are easily seen to be equivalent.

In joint work, Kameryn Williams and I provided the analogous result for a universal $\Sigma_1$ definable sequence. The natural extension concept for $\Sigma_1$ is end extension rather than top extension, since $\Sigma_1$ statements are absolute to top extensions of models of \ZF.

\begin{theorem}[Hamkins,Williams \cite{HamkinsWilliams2021:The-universal-finite-sequence}]\label{Theorem.Sigma_1-universal}
For any c.e. theory $\ZFbar$ extending \ZF, there is a $\Sigma_1$-definable finite sequence
$$a_0,\ a_1,\ \ldots,\ a_n$$
with the following properties:
\begin{enumerate}
\item $\ZF$ proves that the sequence is finite, and it is empty in any transitive model of $\ZFbar$.
\item If $M$ is a countable model of $\ZF$ with an inner model $W\satisfies\ZFbar$ and the sequence defined in $M$ is $s$, then for any finite extension $t$ in $M$, there is a covering end extension of $M$ to a model $N \models \ZFbar$ in which the sequence is exactly $t$.
\end{enumerate}
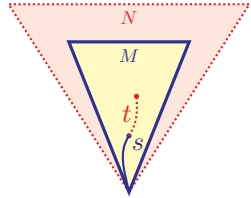
$$\begin{tikzpicture}[xscale=.8,scale=.5]
% extension N
 \draw[thick,red,densely dotted,fill=red!10] (0,0) -- (-4,5) -- node[below,scale=.7] {$N$} (4,5) -- cycle;
% M
 \draw[very thick,blue,fill=Yellow!30] (0,0) -- (-2,4) --  node[below,blue,scale=.7] {$M$} (2,4) -- cycle;
% draw s & t
 \draw[thick,blue] (0,0) to[out=105,in=-130] (0,1.5) node[blue,circle,fill=blue,scale=.2,label={[label distance=-5pt]below right:$s$}] {};
 \draw[thick,densely dotted,red] (0,1.5) to[out=50,in=-85] (.25,2.55) node[red,circle,fill=red,scale=.2,label={[label distance=-3pt]185:$t$}] {};
\end{tikzpicture}$$
\end{theorem}

The theorem exhibits in statement (2) what I call the resurrection phenomenon, by which the theory of an inner model is realized in an end extension. For example, if we start with a model with a measurable cardinal, but then perform forcing to collapse this cardinal, thereby ruining it, we can nevertheless find an end extension to a model of $\ZFC+V=L[\mu]$, which has a measurable cardinal again and even the fine structure of $L[\mu]$. And more generally, any statement true in some inner model can become true again in an end extension. I find it incredible and wonderfully mysterious.

\section{The universal $\Sigma_{m+1}$ definable sequence}

I should like now to provide the set-theoretic generalization of theorem \ref{Theorem.Universal-algorithm-Sigma-m}, which will be a technical improvement on the main results of \cite{HamkinsWoodin:The-universal-finite-set} and \cite{HamkinsWilliams2021:The-universal-finite-sequence}, continuing in the set-theoretic progression of theorems  \ref{Theorem.Sigma_2-universal} and \ref{Theorem.Sigma_1-universal}. Specifically, I shall provide a $\Sigma_{m+1}$-definable finite sequence, with the universal extension property with respect to $\Sigma_m$-elementary end extensions. Ultimately I shall aim to use this theorem to generalize the pointwise-definable end extension argument of main theorem 1 to the case of set theory, proving main theorem 2.

\begin{theorem}\label{Theorem.Sigma_m-universal-ordinals}
For any c.e. theory $\ZFCbar$ extending $\ZFC$ and any natural number $m$, there is a $\Sigma_{m+1}$ definable finite sequence of ordinals
$$\alpha_0,\ \alpha_1,\ \ldots,\ \alpha_n$$
with the following properties:
\begin{enumerate}
\item $\ZF$ proves that the sequence is finite, and it is empty in any transitive model of $\ZFCbar$.
\item If the sequence is $s$ in a countable model $M\satisfies\ZFCbar$, then for every finite sequence $t$ extending $s$ in $M$ there is a $\Sigma_m$ elementary end extension $M\elesub_{\Sigma_m} N\satisfies\ZFCbar$ such that the sequence is $t$ in the extension $N$. (When $m\geq 1$, this is therefore a top extension.)
\end{enumerate}
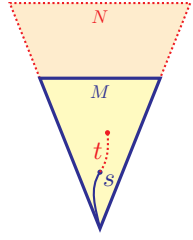
$$\begin{tikzpicture}[xscale=.8,scale=.5]
% extension N
 \draw[thick,red,densely dotted,fill=Orange!20] (-2,4) -- (-3,6) -- node[below,scale=.7] {$N$} (3,6) -- (2,4);
% M
 \draw[very thick,blue,fill=Yellow!30] (0,0) -- (-2,4) --  node[below,blue,scale=.7] {$M$} (2,4) -- cycle;
% draw s & t
 \draw[thick,blue] (0,0) to[out=105,in=-130] (0,1.5) node[blue,circle,fill=blue,scale=.2,label={[label distance=-5pt]below right:$s$}] {};
 \draw[thick,densely dotted,red] (0,1.5) to[out=50,in=-85] (.25,2.55) node[red,circle,fill=red,scale=.2,label={[label distance=-3pt]185:$t$}] {};
\end{tikzpicture}$$
\end{theorem}

\begin{proof}
It will suffice to verify the adding-one extension property, where $t=s\concat\alpha$ extends the sequence $s$ by adding just one ordinal. The reason is that if there is a definable sequence with the adding-one extension property, then by interpreting those ordinals as encoding finite sequences to be concatenated, we can derive a new definable sequence with the universal extension property as stated.

As in both~\cite{HamkinsWoodin:The-universal-finite-set,HamkinsWilliams2021:The-universal-finite-sequence}, I shall describe two set-theoretic processes, $A$ and $B$. These processes are intended to be run inside, respectively, $\omega$-nonstandard models and $\omega$-standard models and will then be merged in a way that provides a single definition fulfilling the desired properties. We may assume $m\geq 1$, since the $\Sigma_1$ definable universal finite sequence of \cite{HamkinsWilliams2021:The-universal-finite-sequence} already fulfills the case $m=0$, as every end extension is $\Sigma_0$-elementary. Fix a specific enumeration of the $\ZFCbar$ theory, and let $\ZFCbar_k$ refer to the theory arising from the first $k$ enumerated axioms.

\textbf{Process $A$.} Let me start by describing process $A$, intended to be run inside $\omega$-nonstandard models as an internal process, using whatever nonstandard natural numbers are to be found there. The process proceeds in a sequence of stages, placing one object onto the sequence at each successful stage. Stage $n$ succeeds and $\alpha_n$ is defined, if there is a $\Sigma_m$-correct cardinal $\delta_n$, larger than $\alpha_n$ and all earlier $\delta_i$, and a natural number $k_n$, smaller than all earlier $k_i$, such that the structure $\<V_{\delta_n},\in>$ has no covering $\Sigma_m$-elementary end extension to a model $\<N,\in^N>$ satisfying $\ZFCbar_{k_n}$, inside which process $A$ places this very ordinal $\alpha_n$ onto the sequence at stage $n$ as the next and last element. Note that the earlier stages of the sequence are absolute to $V_{\delta_n}$ since it is correct about $\Sigma_m$-correctness and about whether a given $V_{\delta_i}$ has a $\Sigma_m$-elementary covering end extension to a model of a certain theory. In slogan form: we place an ordinal onto the sequence, if we find a $\Sigma_m$ correct rank initial segment of the universe having no covering $\Sigma_m$-elementary end extension in which we would have done so as the next and last element. The process officially accepts and uses the triple $(\delta_n,\alpha_n,k_n)$ occurring lexically least. In other words, we first minimize $\delta_n$, such that for some $\alpha_n<\delta_n$ there is such a $k_n$ for which the requirement is fulfilled.\goodbreak

The map $n\mapsto\alpha_n$, I claim, is $\Sigma_{m+1}$ definable. To see this, note first that the property ``$V_\delta$ is $\Sigma_m$ correct'' has complexity $\Pi_m$, since it is equivalent to the assertion (using the universal $\Pi_m$-truth predicate) that every $\Pi_m$ assertion true in $V_\delta$ is actually true. The relation $y=V_\delta$ is $\Pi_1$ in $y$ and $\delta$, and to say that $V_\delta$ has no end extension of the required form is a further $\Pi_1$ assertion. Thus, to verify $\alpha_n$, we ask for the existence of a cardinal $\delta$, an ordinal $\alpha$, a natural number $k$, and a set $y$ such that $y=V_\delta$ and $y$ is $\Sigma_m$ correct, such that $y$ has no end extension to a structure satisfying a certain theory, and all lexically smaller $(\delta',\alpha',k')$ do give rise to such extensions (and indeed, they do so inside $V_\delta$). This is altogether an existential followed by a $\Pi_m$ assertion, and hence has complexity $\Sigma_{m+1}$.

Although the definition may at first appear circular---we define $\alpha_n$, after all, by reference to the definition of the process $A$ sequence inside $N$---one may use the \Godel--Carnap fixed point lemma to find a definition $\psi(n,\alpha)$ for the map $n\mapsto \alpha_n$ that solves this recursion, allowing $\psi$ as described to refer to its own \Godel\ code $\gcode{\psi}$. The same method is used in~\cite{HamkinsWoodin:The-universal-finite-set,HamkinsWilliams2021:The-universal-finite-sequence}, and one should view this as analogous to the common use of the Kleene recursion theorem in computability-theoretic arguments, including its use in the universal algorithm~\cite{Woodin2011:A-potential-subtlety-concerning-the-distinction-between-determinism-and-nondeterminism, Hamkins:The-modal-logic-of-arithmetic-potentialism}.

Since the natural numbers $k_n$ are descending, there will be only finitely many successful stages, and so the sequence will be finite.

I claim that in the relevant models, every $k_n$ arising from a successful stage is nonstandard; in particular, $\omega$-standard models (those having no nonstandard natural numbers) will have no successful process $A$ stages. To see this, consider any countable model of set theory $M\satisfies\ZFCbar$, and let $n$ be the last successful stage in $M$. For any standard $k$, the theory $\ZFCbar_k$ reflects to some $\Sigma_{m+1}$ correct cardinal $\delta$ in $M$ above $\delta_n$. Consequently, $\<V_\delta^M,\in^W>\satisfies \ZFCbar_k$ and thinks object $\alpha_n$ was added at stage $n$ as the last object on the sequence, since this is true in $M$. Furthermore, $V_\delta^M$ is a $\Sigma_m$-elementary end extension of $V_{\delta_n}^M$, as both are $\Sigma_m$-correct in $M$. If $k_n \leq k$, this would violate the success of stage $n$ in $M$, since $V_{\delta_n}$ would have had such a topped end extension after all. Therefore $k<k_n$ for every standard $k$, and so $k_n$ is nonstandard. In particular, in any transitive model of $\ZFCbar$, the sequence will have no successful stages.

It remains to verify the extension property. Let $M$ be an $\omega$-nonstandard countable model of $\ZFCbar$ in which the sequence is $s$. Let $n$ be the first unsuccessful stage.
Let $k$ be nonstandard, but smaller than all previous $k_i$. Since stage $n$ did not succeed, it follows that for any ordinal $\alpha$ and any $\Sigma_m$ correct cardinal $\delta$ above $\alpha$ and all earlier $\delta_i$, the structure $\<V_\delta,\in>^M$ does have a $\Sigma_m$-elementary topped end extension in $M$ to a model satisfying $\ZFCbar_k +$``ordinal $\alpha$ was placed onto the sequence at stage $n$, the last successful stage,'' since otherwise stage $n$ would have succeeded. By the Keisler-Morley theorem, there is an elementary end extension $M\elesub M^+$, and so $n$ is also the first unsuccessful stage in $M^+$ and the two models agree on the outcome of the earlier stages. Let $\delta$ be a $\Sigma_m$ correct cardinal of $M^+$ above $M$. Thus, by what I've just explained, for every ordinal $\alpha<\delta$, in particular every ordinal $\alpha$ in $M$, the structure $\<V_\delta,\in>^{M^+}$ has a $\Sigma_m$ elementary covering end extension $\<N,\in^N>\satisfies\ZFCbar_k$ in which the $A$ sequence adds the ordinal $\alpha$ at stage $n$ as the last successful stage. Since $\delta$ is $\Sigma_m$ correct in $M^+$, it is a $\Sigma_m$-elementary end extension of $M$, and so $N$ is a $\Sigma_m$-elementary end extension of $M$ realizing $\alpha$ as the next and last element on the sequence, as desired.

\textbf{Process $B$.} Let me turn now to process $B$, intended to be run as an internal process inside $\omega$-standard models, using whatever (possibly nonstandard) ordinals are to be found there. The process proceeds in a sequence of stages, just as before, with each successful stage adding one ordinal $\alpha_n$ to the sequence. Stage $n$ is successful and $\alpha_n$ is defined, if there is a $\Sigma_m$ correct cardinal $\delta_n$ above $\alpha_n$ and all earlier $\delta_i$ and a countable ordinal $\lambda_n$, smaller than all earlier $\lambda_i$, such that the structure $\<V_{\delta_n},\in>^M$ has no covering $\Sigma_m$-elementary end extension to a model $\<N,\in^N>$ satisfying $\ZFCbar+$``process $B$ places ordinal $\alpha_n$ on the sequence at stage $n$, the last successful stage,'' and furthermore, it has no such end extension even in the forcing extension of $M$ collapsing $V_{\delta_n}$ to become countable, and finally, the assertion about the real coding $V_{\delta_n}$ and $\alpha_n$ that there is no such end extension in that forcing extension, which is a $\Pi^1_1$ statement, has rank $\lambda_n$ in the canonical representation of $\Pi^1_1$-assertions by well-founded trees.

Once again, the circularity in the definition can be removed by the \Godel--Carnap fixed-point lemma. Since the ordinals $\lambda_n$ are descending, there will be only finitely many successful stages, and so the sequence is finite.

And once again, the map $n\mapsto \alpha_n$ is $\Sigma_{m+1}$ definable, since one asks for a cardinal $\delta_n$ and ordinals $\alpha_n$ and $\lambda_n$, and a set $y$, such that $y=V_\delta$ is $\Sigma_m$ correct, and furthermore such that it is forced that this set has no end extension in the collapse forcing to a model of a certain theory, and further the rank of that $\Pi^1_1$ assertion is $\lambda_n$. Except for the $\Sigma_m$ correctness, which has complexity $\Pi_m$, the rest of this can be verified inside any sufficiently large rank-initial segment of the universe and hence has complexity $\Sigma_2$, making the complexity $\Sigma_{m+1}$ altogether.

I claim that the ordinals $\lambda_n$ in the relevant models are nonstandard. Suppose that $M$ is a countable model of $\ZFCbar$  and let $n$ be the last successful stage in $M$. This stage was successful because there was a $\Sigma_m$ correct cardinal $\delta_n$ and ordinal $\alpha_n$, such that in the forcing extension collapsing $V_{\delta_n}$ the structure $\<V_{\delta_n},\in>^M$ had no $\Sigma_m$-elementary covering end extension to a model $\<N,\in^N>$ in which the $B$ sequence added $\alpha_n$ as the next and final stage, and such that the rank of this assertion is $\lambda_n$. But since $M$ itself is such a covering $\Sigma_m$ elementary end extension of $V_{\delta_n}^M$ in which $\alpha_n$ was added as the last element of the sequence, the statement that there is no such end extension is not actually true. Thus, the $\Pi^1_1$ assertion in the forcing extension of $M$ making $V_{\delta_n}$ countable is not actually true, even though this model thought it was true. Since this is an $\omega$-standard model, the tree it builds to represent this $\Pi^1_1$ statement is the same as we would build external to $M$, and so the tree is not actually well-founded, although it was found to have rank $\lambda_n$ inside the model. So $\lambda_n$ must be nonstandard. Consequently, the sequence is empty in any transitive model of $\ZFCbar$.

\enlargethispage{25pt}
Let me finally prove the extension property for process $B$ in countable $\omega$-standard models $M\satisfies\ZFCbar$. Suppose the sequence is $s$ in $M$. Let $n$ be the first unsuccessful stage, and consider any ordinal $\alpha$ in $M$. By the Keisler--Morley theorem, there is a countable elementary end extension of $M$ to a model $M^+$. Let $\delta$ be a $\Sigma_m$ correct cardinal in $M^+$ above $M$, and let $M^+[G]$ be a forcing extension in which $V_\delta^{M^+}$ is made countable. If $\<V_\delta,\in>^{M^+}$ has a $\Sigma_m$-elementary covering end extension to a model of $\ZFCbar$ in which ordinal $\alpha$ is placed onto the sequence at stage $n$ as the last successful stage, then we'd be done, since this would also serve as the desired end extension of the original model $M$. So let me assume toward contradiction that there is no such end extension of $\<V_\delta,\in>^{M^+}$. This is a $\Pi^1_1$-assertion about the generic real coding $V_\delta^{M^+}$ and $\alpha$ in $M^+[G]$, and this $\Pi^1_1$ statement therefore has some ordinal rank $\lambda$ there in the representation of $\Pi^1_1$-assertions by well-founded trees. Since the statement is really true (in the set-theoretic background of our universe), it follows that $\lambda$ must be in the well-founded part of $M^+[G]$. In particular, $\lambda<\lambda_i$ all $i<n$, since those ordinals are nonstandard. But this implies that stage $n$ would have been successful in $M^+$, contrary to assumption. So $V_\delta^{M^+}$ must have had the desired end extension after all, and so I've verified the extension property of process $B$ for the relevant countable $\omega$-standard models of set theory.

{\bf Process C.} I shall now merge processes $A$ and $B$ into a single process $C$, providing a single uniform $\Sigma_{m+1}$-definition that will work in all the relevant countable models of set theory. Process $C$ proceeds in stages. At each successful stage, it will place one new object onto the sequence, either for an $A$-reason or for a $B$-reason. The $A$-reasons will involve data $(\delta_n,\alpha_n,k_n)$ as in process $A$, and the $B$-reasons will involve data $(\delta_n,\alpha_n,\lambda_n)$ as in process $B$, and in each case we shall insist that $\delta_n$ is a $\Sigma_m$ correct cardinal above all the other data and the earlier $\delta_i$ and furthermore that $k_n$ is below all earlier $A$-reason stage $k_i$, if this is an $A$-reason stage, and $\lambda_n$ is below all earlier $B$-reason stage $\lambda_i$, if this is a $B$-reason stage. The $A$ reason stages are successful, if $\<V_{\delta_n},\in>$ has no covering $\Sigma_m$-elementary end extension to a model $\<N,\in^N>\satisfies\ZFCbar_{k_n}$ in which process $C$ adds $\alpha_n$ as the next and last element; the $B$ reason stages are successful, if $\<V_{\delta_n},\in>$ has no covering $\Sigma_m$-elementary end extension to a model $\<N,\in^N>\satisfies\ZFCbar$ in which process $C$ places $\alpha_n$ as the next and last element on the sequence, and furthermore there is no such extension even in the forcing extension making $V_{\delta_n}$ countable and the rank of this $\Pi^1_1$ assertion is $\lambda_n$. Once there has been a successful $A$-reason stage, then we proceed further only with $A$-reason stages.

To complete the argument, we observe that the analysis of processes $A$ and $B$ goes through for process $C$. The map $n\mapsto\alpha_n$ is $\Sigma_{m+1}$ definable for similar reasons as in processes $A$ and $B$. The sequence is finite since the $k_n$ and $\lambda_n$ can go down only finitely many times. At any $A$-reason stage, the number $k_n$ must be nonstandard, and at any $B$-reason stage, the ordinal $\lambda_n$ must be nonstandard, and so the sequence is empty in any transitive model of $\ZFCbar$. In any $\omega$-nonstandard model, we achieve the extension property for process $C$ for the same reasons as process $A$; and in any $\omega$-standard model, we achieve the extension property for process $C$ just as with process $B$.
\end{proof}

In the previous theorem, the universal sequence consists of ordinals only. In the case that we have a model of $V=\HOD$, however, every object is definable as the $\alpha$th element for some ordinal $\alpha$. Thus we can achieve a form of the universal finite sequence for arbitrary objects.

\begin{theorem}\label{Theorem.Sigma_m-universal-arb-sets}
For any c.e. theory $\ZFCbar$ extending $\ZFC+V=\HOD$ and any natural number $m$, there is a $\Sigma_{m+1}$ definable finite sequence of sets
$$a_0,\ a_1,\ \ldots,\ a_n$$
with the following properties:
\begin{enumerate}
\item $\ZF$ proves that the sequence is finite, and it is empty in any transitive model of $\ZFCbar$.
\item If the sequence is $s$ in a countable model $M\satisfies\ZFCbar$, then for every finite sequence $t$ extending $s$ in $M$ there is a $\Sigma_m$ elementary end extension $M\elesub_{\Sigma_m} N\satisfies\ZFCbar$ such that the sequence is $t$ in the extension $N$.
\end{enumerate}
\end{theorem}

\begin{proof}
If $V=\HOD$ is part of the theory $\ZFCbar$, then this theorem is a consequence of theorem \ref{Theorem.Sigma_m-universal-ordinals}, since from the universal finite sequence of ordinals, we can use the definable well-ordering of the universe to pick out for each ordinal $\alpha$ the $\alpha$th set in that definable order. The $\alpha$th element of $\HOD$ is $\Delta_2$ definable from $\alpha$, and so from the definable ordinal sequence we extract the sequence of definable sets. And since we have the universal extension property for $\Sigma_m$-elementary end extensions for the ordinal sequence, we can thereby put any desired set object as the next element on the derived sequence, thereby establishing this theorem.
\end{proof}

Corollary \ref{Corollary.V=HOD-required} of the main theorem will show that the $V=\HOD$ part of this theorem cannot be dropped. Specifically, it is not possible to provide a definable universal finite sequence with the universal extension property for $\Sigma_m$-elementary end extensions in the countable models of a theory that does not prove $V=\HOD$.

\section{Pointwise-definable end extensions in set theory}

I am now able to prove main theorem 2 by applying the underlying method of theorem 1, but doing so in set theory with the generalized version of the universal finite sequence provided by theorem \ref{Theorem.Sigma_m-universal-ordinals}.

\begin{theorem}[Main theorem 2]\label{Theorem.Main-theorem-2-numbered}
  Every countable model of \ZF\ has a pointwise-definable end extension satisfying $\ZFC+V=L$.
\end{theorem}

A somewhat more general version of the theorem is the following:

\begin{theorem}
 Every countable model of \ZF\ with an inner model of a c.e. theory $\ZFCbar$ that includes $V=\HOD$ has a pointwise-definable end extension satisfying $\ZFCbar$.
\end{theorem}

This theorem in turn is a consequence of the following somewhat more refined version, dropping the $V=\HOD$ requirement, but obtaining still a \emph{Paris} model, one in which every ordinal is definable without parameters (see \cite{Enayat2005:ModelsOfSetTheoryWithDefinableOrdinals}). The point is that every Paris model of $V=\HOD$ is pointwise definable, since the ordinal parameters are themselves definable.

\begin{theorem}
  Every countable model of \ZF\ with an inner model of a c.e. theory $\ZFCbar$ has an end extension to a model of $\ZFCbar$ that is a Paris model.
\end{theorem}

This theorem can be proved by first appealing to theorem \ref{Theorem.Sigma_1-universal} to get any end extension of the original model to a model of $\ZFCbar$, and then applying the following theorem to find a further extension to a pointwise-definable model.

\begin{theorem}\label{Theorem.Sigma_m-elementary-end extension} Let $m$ be any particular finite natural number.
 \begin{enumerate}
   \item Every countable model $M$ of a c.e. theory $\ZFCbar$ extending $\ZFC$ has a $\Sigma_m$ elementary end extension
     $$M\elesub_{\Sigma_m} N$$
     to a Paris model $N\satisfies\ZFCbar$, that is, one in which every ordinal is definable without parameters.
   \item Consequently, every countable model $M$ of a c.e. theory $\ZFCbar$ extending $\ZFC+V=\HOD$ has a $\Sigma_m$-elementary end extension $M\elesub_{\Sigma_m}N$ to a pointwise-definable model $N\satisfies\ZFCbar$.
 \end{enumerate}
\end{theorem}

\begin{proof}
Fix any countable model $M_0$ of $\ZF$. By theorem \ref{Theorem.Sigma_m-universal-ordinals}, we can find a $\Sigma_m$-elementary end extension model $M_1$ where the last entry of the universal $\Sigma_{m+1}$-sequence is any desired ordinal $\alpha_0$ of $M_0$. Continuing iteratively just as in theorem \ref{Theorem.Main-theorem-1-numbered}, we apply theorem \ref{Theorem.Sigma_m-universal-ordinals} again to find a $\Sigma_{m+1}$-elementary countable end extension $M_2$, such the last element on the universal $\Sigma_{m+2}$-sequence is any desired ordinal $\alpha_1$ of $M_1$. Note that the $\Sigma_{m+1}$-elementarity means that we have preserved the previous $\Sigma_{m+1}$ definable universal sequence, and so $\alpha_0$ is still last on that sequence. And so on. At each stage $n$, we find by theorem \ref{Theorem.Sigma_m-universal-ordinals} a $\Sigma_{m+n}$ elementary end extension $M_{n+1}$ of $M_n$, such that any desired ordinal $\alpha_n$ from $M_n$ is last on the $\Sigma_{m+n+1}$-definable universal sequence, and this will be preserved to the later models. The result is an progressively elementary tower of countable models:
 $$M_0\quad\elesub_{\Sigma_m}\quad M_1\quad\elesub_{\Sigma_{m+1}}\quad M_2\quad\elesub_{\Sigma_{m+2}}\quad M_3\quad\elesub_{\Sigma_{m+3}}\quad M_4\quad\elesub_{\Sigma_{m+4}}\quad \cdots$$
Because the tower is progressively elementary, each model $M_n$ will be $\Sigma_{m+n}$ elementary in the union model $M^*=\Union_n M_n$, which therefore will be a model of the desired theory $\ZFCbar$, as well as $\Sigma_m$-elementary extension of the original model $M_0$. Every ordinal $\alpha_n$ is definable in $M^*$ as the last element on the $\Sigma_{m+n+1}$-definable universal sequence, since this is true in $M_{n+1}$ and preserved by all the later extensions. Therefore, since these are all countable models, a simple bookkeeping method will suffice for us to arrange that the elements $\alpha_n$ exhaust the ordinals of $M^*$, which will therefore be the desired target Paris model, fulfilling statement (1).

For statement (2), we simply observe that every Paris model of $V=\HOD$ is pointwise definable, since every object is definable with ordinal parameters, but the ordinals themselves are definable without parameters.
\end{proof}

\begin{corollary}\label{Corollary.V=HOD-required}
 If a theory $\ZFCbar$ supports universal definable finite sequence of arbitrary sets with the universal extension property with respect to $\Sigma_m$-elementary end extensions of countable models of $\ZFCbar$, for every $m$, then it proves $V=\HOD$.
\end{corollary}

\begin{proof}
 Suppose that for every $m$ we had a definable universal finite sequence with the universal extension property for adding any particular set to the sequence in some $\Sigma_m$-elementary end extension of a given countable model of $\ZFCbar$. Let $M_0$ be any countable model of $\ZFCbar$. Using the method of main theorem 2 (theorem \ref{Theorem.Main-theorem-2-numbered}), we can build a progressively elementary tower of models
  $$M_0\elesub_{\Sigma_{m_0}}M_1\elesub_{\Sigma_{m_1}}M_2\elesub_{\Sigma_{m_2}}\cdots \elesub_{\Sigma_2} M$$
 where $m_0\geq 2$ and at each stage we impose a degree of elementarity $\Sigma_{m_n}$ exceeding the complexity of the definitions of the previous sequences, and where furthermore by bookkeeping we arrange that the union model $M$ is pointwise definable. Since the tower is progressively elementary, the limit model $M$ is a model of $\ZFCbar$. Since it is pointwise definable, it satisfies $V=\HOD$. Since $M_0\elesub_{\Sigma_2} M$, it follows that $M_0$ is also a model of $V=\HOD$, since this statement has complexity $\Pi_3$. So every model of $\ZFCbar$ must be a model of $V=\HOD$, and the corollary is proved.
 \end{proof}

This is a sense in which $V=\HOD$ is required for the $\Sigma_m$-elementary universal extension property.

\section{Application to $\Sigma_m$-elementary potentialism}

Let me briefly discuss the consequences of the $\Sigma_m$-elementary universal finite sequence for $\Sigma_m$-elementary potentialism. We consider the collection of countable models of \ZFC\ under the relation of $\Sigma_m$-elementary end extensions. (These are all top extensions when $m\geq 1$.) This is a potentialist system in the sense of \cite{HamkinsLinnebo2022:Modal-logic-of-set-theoretic-potentialism, HamkinsWoodin:The-universal-finite-set, HamkinsWilliams2021:The-universal-finite-sequence}, and supports the modal semantics by which $\possible\psi$ is true in a model $M$ if there is some $\Sigma_m$-elementary end extension $M\elesub_{\Sigma_m}N$ for which $N\satisfies\psi$, and $\necessary\psi$ is true in $M$ if $\varphi$ is true in all such $N$. A statement $\varphi(p_1,\ldots,p_n)$ of propositional modal logic is \emph{valid} for this potentialist system, with respect to a language of substitution instances, if $\varphi(\psi_1,\ldots,\psi_n)$ holds for the potentialist modal semantics for all such substitution instances. A central concern with any potentialist system is to discover the modal validities, and here I do this for the potentialist system at hand.

\begin{theorem}
 The modal validities true in the countable models of \ZFC\ under $\Sigma_m$-elementary end extensions form exactly the modal theory $\theoryf{S4}$.
\end{theorem}

\begin{proof}
The modal theory \theoryf{S4} is valid in any potentialist system. For the converse, it suffices as in \cite[theorem~30]{Hamkins:The-modal-logic-of-arithmetic-potentialism}, \cite[theorem~8]{HamkinsWoodin:The-universal-finite-set}, or \cite[theorem~11]{HamkinsWilliams2021:The-universal-finite-sequence}, to provide a railyard labelings of the finite tree pre-orders. This can be undertaken with the $\Sigma_{m+1}$-definable universal finite sequence, just as in those arguments with the analogous universal sequences. Namely, one interprets the objects on the universal sequence as instructions for climbing the finite tree preorder, and this will align $\Sigma_m$-elementary end-extensional possibility with accessibility in the tree preorder. It follows that the potentialist validities will be contained within the validities of the finite tree preorder frames, and this is \theoryf{S4}, as desired.
\end{proof}

\begin{theorem}
 No model of \ZF\ has a maximal $\Sigma_{m+1}$-theory (allowing natural number parameters), even with respect to fixing the $\Sigma_m$ diagram of the model. That is, new $\Sigma_{m+1}$ facts can become true about natural numbers in a $\Sigma_m$-elementary end extension.
\end{theorem}

\begin{proof}
This is immediate from the $\Sigma_{m+1}$-definable universal extension property, since the universal sequence can be extended to have one additional successful stage, which is a new $\Sigma_{m+1}$ fact that becomes true in a $\Sigma_m$-elementary end extension.
\end{proof}

Meanwhile, a simple compactness maximizing argument on the theory produces models of \ZF\ with a maximal $\Sigma_{m+1}$ theory, when parameters are not allowed.

\section{Philosophical discussion}

Let me conclude with a few philosophical remarks. The \emph{toy model} perspective in set theory, as undertaken in \cite{Hamkins2012:TheSet-TheoreticalMultiverse, Hamkins2014:MultiverseOnVeqL, Hamkins2015:IsTheDreamSolutionToTheContinuumHypothesisAttainable, GitmanHamkins2010:NaturalModelOfMultiverseAxioms}, is the proposal that we might learn about the nature of set-theoretic possibility beyond $V$, a realm otherwise inaccessible to us, by studying the collection of countable models of set theory. The idea is that for all we know, the current set-theoretic universe $V$ in which we find ourselves might be a mere countable model inside some other much better universe $V^+$ that we cannot reach or fathom. In this way, the multiverse of $V$ becomes a part of the toy multiverse as it is seen from $V^+$, thus falling under the general analysis of such toy models. We can certainly observe this phenomenon abundantly amongst the smaller worlds to which we do have fuller access, where models that seem uncountable in some contexts are ultimately seen as countable in larger more powerful contexts. By studying the countable models as a toy simulacra of the larger multiverse in which we are truly interested, therefore, we may come to insight about it.

The toy model perspective reveals set theory as providing a vast generalization of the object-theory/meta-theory distinction---the models provide a hierarchy of metatheoretic contexts. Each model serves as a metatheoretic context for the models and theories available in it, and these various contexts can disagree about core metamathematical facts, such as questions about consistency, satisfiability, and consistency strength. In this way, metamathematical questions are realized as mathematical in the higher model. Questions about the set-theoretic universe might begin as sublime and philosophical, becoming metamathematical in connection with a specific model of set theory $M$, explicitly so in the metamathematical context provided by a model $M^+$ in which $M$ exists as a set, but ultimately seen as purely mathematical in the object theory of $M^+$. This model $M^+$ can itself be considered from the metatheoretic context provided by a still-higher model $M^{++}$, which may look upon both $M$ and $M^+$ as countable toy models. Ascending the hierarchy of metatheoretic contexts, we thus move from philosophy to metamathematics to mathematics to countable mathematics.

Large cardinal set theory often exhibits a certain potentialist nature not dissimilar to this. Namely, set theorists commonly specify a large cardinal theory without any suggestion of finality, but rather, freely adopt much stronger large cardinal hypotheses whenever the need arises. A set theorist may start with a proper class of Woodin cardinals, but then want a supercompact cardinal, or an extendible cardinal, or a proper class of superhuge cardinals, or what have you. The actual large cardinal theory in play thus has an upwardly extensible or potentialist aspect---it can become stronger at any moment for any reason. In this way, the standard use of large cardinals in set theory exhibits a potentialist character.

What the theorems of this article show is that certain set-theoretic features occur upwardly densely often in the toy multiverse of countable models. Namely, every countable model of \ZF\ has an end extension to a pointwise-definable model of $V=L$, and also to models that are not pointwise definable. The feature of being pointwise definable is thus a switch in the potentialist modal logic, a statement that can be turned on and off as one accesses larger worlds. In light of this, the toy model perspective thereby resurrects the math tea argument---shall we think of the full set-theoretic universe $V$ as being extendible to a pointwise-definable universe from the perspective of a still-larger metamathematical context?

\printbibliography

\end{document}

%% file: MathMacrosJDH.tex
\newtheorem{theorem}{Theorem}
\newtheorem*{theorem*}{Theorem}
\newtheorem{maintheorem}[theorem]{Main Theorem}

\newtheorem*{maintheorem*}{Main Theorem}
\newtheorem*{maintheorems*}{Main Theorems}
\newtheorem{corollary}[theorem]{Corollary}
\newtheorem*{corollary*}{Corollary}
\newtheorem*{corollaries*}{Corollaries}

\theoremstyle{definition}

\newtheorem*{definition*}{Definition}

\newtheorem*{question*}{Question}

\newtheorem*{questions*}{Questions}
\newtheorem*{mainquestion*}{Main Question} % without numbering
\newtheorem*{openquestion*}{Open Question} % without numbering
\theoremstyle{remark}

\newcommand{\QED}{\end{proof}}

\def\proclaim[#1]{{\bf #1}}
\def\BF#1.{{\bf #1.}}

\def\says#1:#2\par{\item[#1] #2\par}
\newcommand{\Godel}{G\"odel}

\newcommand{\N}{{\mathbb N}}

\makeatletter
\newcommand{\dotminus}{\mathbin{\text{\@dotminus}}}
\newcommand{\@dotminus}{%
  \ooalign{\hidewidth\raise1ex\hbox{.}\hidewidth\cr$\m@th-$\cr}%
}
\makeatother

\newcommand{\of}{\subseteq}

\newcommand{\elesub}{\prec}

\newcommand{\satisfies}{\models}

\DeclareMathOperator{\possible}{\text{\tikz[scale=.6ex/1cm,baseline=-.6ex,rotate=45,line width=.1ex]{\draw (-1,-1) rectangle (1,1);}}}
\DeclareMathOperator{\necessary}{\text{\tikz[scale=.6ex/1cm,baseline=-.6ex,line width=.1ex]{\draw (-1,-1) rectangle (1,1);}}}

\newcommand\dbrace{\hskip-1.5em\raise3pt\hbox{\rotatebox[origin=c]{-35}{$\left.\strut^{\phantom{|}}\right\}$}}}% useful for tetration

\newcommand\UParroW{{\setbox0\hbox{$\Uparrow$}\rlap{\hbox to \wd0{\hss$\mid$\hss}}\box0}}

\newcommand{\theoryf}[1]{{\rm #1}}% {\hbox{$\mathsf{#1}$}}

\newcommand{\concat}{\mathbin{{}^\smallfrown}}

\renewcommand{\setminus}{\raise.3ex\hbox{\rotatebox{-20}{$-$}}} % the usual setminus is absurdly huge and vertical

\renewcommand{\emptyset}{\varnothing}

\newcommand{\Union}{\bigcup}

\newcommand{\smalllt}{\mathrel{\mathchoice{\raise2pt\hbox{$\scriptstyle<$}}{\raise1pt\hbox{$\scriptstyle<$}}{\raise0pt\hbox{$\scriptscriptstyle<$}}{\scriptscriptstyle<}}}
\newcommand{\smallleq}{\mathrel{\mathchoice{\raise2pt\hbox{$\scriptstyle\leq$}}{\raise1pt\hbox{$\scriptstyle\leq$}}{\raise1pt\hbox{$\scriptscriptstyle\leq$}}{\scriptscriptstyle\leq}}}

   \def\DHLhksqrt#1#2{%
   \setbox0=\hbox{$#1\sqrt{#2\,}$}\dimen0=\ht0
   \advance\dimen0-0.2\ht0
   \setbox2=\hbox{\vrule height\ht0 depth -\dimen0}%
   {\box0\lower0.4pt\box2}}

\def\[#1]{\mathopen{\lbrack\!\lbrack}#1\mathclose{\rbrack\!\rbrack}}
\newbox\gnBoxA
\newbox\gnBoxB
\newdimen\gnCornerHgt
\setbox\gnBoxA=\hbox{\tiny$\ulcorner$}
\global\gnCornerHgt=\ht\gnBoxA
\newdimen\gnArgHgt
\def\gcode #1{%
\setbox\gnBoxA=\hbox{$#1$}%
\setbox\gnBoxB=\hbox{$\bar #1$}%
\gnArgHgt=\ht\gnBoxB%
\ifnum     \gnArgHgt<\gnCornerHgt \gnArgHgt=0pt%
\else \advance \gnArgHgt by -\gnCornerHgt%
\fi \raise\gnArgHgt\hbox{\tiny$\ulcorner$} \box\gnBoxA %
\raise\gnArgHgt\hbox{\tiny$\urcorner$}}
\newcommand{\UnderTilde}[1]{{\setbox1=\hbox{$#1$}\baselineskip=0pt\vtop{\hbox{$#1$}\hbox to\wd1{\hfil$\sim$\hfil}}}{}}
\newcommand{\Undertilde}[1]{{\setbox1=\hbox{$#1$}\baselineskip=0pt\vtop{\hbox{$#1$}\hbox to\wd1{\hfil$\scriptstyle\sim$\hfil}}}{}}
\newcommand{\undertilde}[1]{{\setbox1=\hbox{$#1$}\baselineskip=0pt\vtop{\hbox{$#1$}\hbox to\wd1{\hfil$\scriptscriptstyle\sim$\hfil}}}{}}
\newcommand{\UnderdTilde}[1]{{\setbox1=\hbox{$#1$}\baselineskip=0pt\vtop{\hbox{$#1$}\hbox to\wd1{\hfil$\approx$\hfil}}}{}}
\newcommand{\Underdtilde}[1]{{\setbox1=\hbox{$#1$}\baselineskip=0pt\vtop{\hbox{$#1$}\hbox to\wd1{\hfil\scriptsize$\approx$\hfil}}}{}}

\renewcommand{\implies}{\mathrel{\rightarrow}}

\def\<#1>{\left\langle#1\right\rangle}

\newcommand\No{\mathord{\mathbb{N}^{\kern-.5pt\stackunder[.6pt]{\tiny$\mathbbm{o}$}{\tikz{\draw[very thin] (.3ex,0) rectangle (1ex,.2ex);}}}}} % for surreals, requires \usepackage{bbm}, \usepackage{stackengine}
\newcommand{\ZFC}{{\rm ZFC}}
\newcommand{\ZF}{{\rm ZF}}

\newcommand{\GBC}{{\rm GBC}}

\newcommand{\GCH}{{\rm GCH}}

\newcommand{\HOD}{{\rm HOD}}

\newcommand{\PA}{{\rm PA}}

\newcommand{\cell}[1]{\boxit{\hbox to 17pt{\strut\hfil$#1$\hfil}}}
\newcommand{\head}[2]{\lower2pt\vbox{\hbox{\strut\footnotesize\it\hskip3pt#2}\boxit{\cell#1}}}
\newcommand{\boxit}[1]{\setbox4=\hbox{\kern2pt#1\kern2pt}\hbox{\vrule\vbox{\hrule\kern2pt\box4\kern2pt\hrule}\vrule}}
\newcommand{\Col}[3]{\hbox{\vbox{\baselineskip=0pt\parskip=0pt\cell#1\cell#2\cell#3}}}
\newcommand{\tapenames}{\raise 5pt\vbox to .7in{\hbox to .8in{\it\hfill input: \strut}\vfill\hbox to
.8in{\it\hfill scratch: \strut}\vfill\hbox to .8in{\it\hfill output: \strut}}}
\newcommand{\Head}[4]{\lower2pt\vbox{\hbox to25pt{\strut\footnotesize\it\hfill#4\hfill}\boxit{\Col#1#2#3}}}
\newcommand{\Dots}{\raise 5pt\vbox to .7in{\hbox{\ $\cdots$\strut}\vfill\hbox{\ $\cdots$\strut}\vfill\hbox{\
$\cdots$\strut}}}